\documentclass[letterpaper,11pt,twoside,keywordsasfootnote,addressatend,noinfoline]{article}
\usepackage{fullpage}
\usepackage[english]{babel}
\usepackage{amssymb}
\usepackage{amsmath}
\usepackage{theorem}
\usepackage{epsfig}
\usepackage{mathrsfs}
\usepackage{imsart}
\usepackage{color}

\theorembodyfont{\slshape}
\newtheorem{theorem}{Theorem}

\newtheorem{lemma}[theorem]{Lemma}

\newenvironment{proof}{\noindent{\scshape Proof.}}{\hspace*{2mm} $\square$}

\newcommand{\C}{\mathscr{C}}

\newcommand{\T}{\mathbb{T}}
\newcommand{\n}{\hspace*{-6pt}}

\DeclareMathOperator{\card}{card \,}
\DeclareMathOperator{\bernoulli}{Bernoulli \,}
\DeclareMathOperator{\var}{Var}
\DeclareMathOperator{\diam}{diam}

\begin{document}

\begin{frontmatter}
\title     {Bidirectional bond percolation model for the \\ spread of information in financial markets}
\runtitle  {Spread of information in financial markets}
\author    {Stefano Chiaradonna and Nicolas Lanchier\thanks{This work is partially supported by NSF grant CNS-2000792.}}
\runauthor {Stefano Chiaradonna and Nicolas Lanchier}
\address   {School of Mathematical and Statistical Sciences \\ Arizona State University \\ Tempe, AZ 85287, USA. \\ schiarad@asu.edu \\ nicolas.lanchier@asu.edu}

\begin{abstract} \ \
 Information is a key component in determining the price of an asset in financial markets, and the main objective of this paper is to study the spread of information in this context.
 The network of interactions in financial markets is modeled using a Galton-Watson tree where vertices represent the traders and where two traders are connected by an edge if one of the two traders sells the asset to the other trader.
 The information starts from a given vertex and spreads through the edges of the graph going independently from seller to buyer with probability~$p$ and from buyer to seller with probability~$q$.
 In particular, the set of traders who are aware of the information is a (bidirectional) bond percolation cluster on the Galton-Watson tree.
 Using some conditioning techniques and a partition of the cluster of open edges into subtrees, we compute explicitly the first and second moments of the cluster size, i.e., the random number of traders who learn about the information.
 We also prove exponential decay of the diameter of the cluster in the subcritical phase.
\end{abstract}



\end{frontmatter}

\maketitle


\section{Introduction}
 A wealth of models have been developed to study financial markets, defined as \emph{any places or systems that provide buyers and sellers the means to trade financial instruments}~\cite{financial_market_def}.
 Financial markets consist in particular of systems where buyers and sellers interact.
 One class of models of financial markets that have gained attraction are the ones based on percolation theory, the branch of probability concerned with the statistical properties of random clusters on graphs~\cite{broadbent_hammersley_1957, grimmett_1999}.
 Popular examples of percolation models used to study financial markets are the Cont-Bouchaud market model and the Ising model.
 The Cont-Bouchaud market model uses percolation on the~$d$-dimensional lattice, where each occupied site represents a trader~\cite{Stauffer_2001}.
 Traders can be partitioned into clusters of the same opinion, and the price variations of the asset is controlled by the size of these clusters~\cite{Stauffer_1999}.
 Similarly, the Ising model is used to study the interactions taking place in financial markets due to its ability to describe the competition between social imitation of the model's individual elements and the impact of private information~\cite{Ising_Model_Economics}. \\
\indent
 Due to the presence of interactions between buyers and sellers, a trader's decision is influenced by the news received from others within the market~\cite{WANG2010431}.
 In particular, information has become an important factor in determining the price of an asset.
 This relates to the concept of price efficiency, defined as \emph{the degree to which prices reflect all available information in terms of speed and accuracy}~\cite{saffi_2011}.
 The value of markets depends on the amount of information about prices, which reveals whether an investor should buy or sell \cite{bond_et_al}.
 The disagreements among investors about the expected prices due to diverse investor information is a common feature in most financial markets~\cite{albagli_2011}, and these differences in the distribution of private information may explain why informational efficiency can vary significantly~\cite{CORGNET2020103671}. \\
\indent
 The literature about the impact of information on financial markets is copious.
 Edmans et al.~\cite{edmans_2017} argued that decisions depend not only on the total amount of information in prices but also on the source of the information.
 Albagli et al.~\cite{albagli_2011} developed a model of asset pricing to demonstrate that the heterogeneity of information and its aggregation within a financial market are core forces in determining asset prices.
 Corgnet et al.~\cite{CORGNET2020103671} proposed that the concentration of private information across the investors is the determinant of information efficiency.
 Lv and Wu~\cite{margin_trading_2020} investigated the effect of margin trading on information content and price-adjustment speed of price efficiency.
 In particular, developing and studying mathematical models for the spread of information in financial markets is of primary importance. \vspace*{5pt} \\
{\bf Galton-Watson random tree.}
 Before describing how the information spreads, the first step is to explain the topological structure of a financial market and how the traders interact.
 Our model starts with an Initial Public Offering~(IPO), which is \emph{when a company first sells its shares to the public}~\cite{ipo}, and the financial market consists of buyers and sellers that consider this IPO as their only asset of interest.
 The IPO sells its shares to some buyers who then sell their shares to other buyers, and so on, and we assume that potential buyers only purchase from a single seller and that sellers are not currently interested in repurchasing any shares of the asset.
 Therefore, the only potential buyers are those who have not already possessed the shares.
 In particular, thinking of each trader as a vertex and drawing an edge between two traders if one of the two traders sells the shares to the other trader results in a graph with no cycles, meaning that the network of interactions among traders exhibits a tree structure.
 To also account for the randomness inherent to a financial market, we assume that the network of interactions is a Galton-Watson tree:
 the graph starts from the~IPO~(the root of the tree) and the random number of buyers of each seller~(the number of edges starting from each vertex going away from the root) follows a fixed distribution.  \vspace*{5pt} \\
{\bf Bidirectional bond percolation.}
 As previously mentioned, information is a key component in the pricing of an asset.
 To model the spread of information within the financial market, we assume that a trader in the already formed random tree has new information about the price of the shares.
 In order to price the shares accurately, this trader and any other traders who learn the information in the future gather more information by inquiring their buyers and/or their seller so that the information spreads through the edges of the graph with some fixed probabilities.
 Due to the asymmetry in the seller-buyer relationship, we assume that the information spreads from a seller to each of her buyers independently with probability~$p$ while it spreads from a buyer to her seller with probability~$q$.
 In mathematical terms, the resulting set of traders who learn the information is an open cluster in a (bidirectional) bond percolation process on the Galton-Watson tree where edges are independently open with probability~$p$ in the seller-buyer direction going away from the root and with probability~$q$ in the buyer-seller direction going towards the root. \vspace*{5pt} \\
 The main objective of this paper is to study some of the statistical properties of the size distribution of the open cluster, i.e., the number of traders who learn about the information.
 More precisely, we compute explicitly the first and second moments of the cluster size, and prove exponential decay of the diameter of the cluster in the subcritical phase.


\section{Model description and main results}
 To define the model rigorously, let~$\T = (V, E)$ be a realization of the Galton-Watson tree with offspring distribution~$(p_k)$, meaning that, independently for each vertex, there are~$k$ edges starting from this vertex and going away from the root of the tree with probability~$p_k$.
 In the terminology of branching processes, two vertices connected by an edge are called the parent and the offspring, with the parent being the vertex closer to the root.
 In the context of financial markets, the parent plays the role of the seller while the offspring plays the role of the buyer.
 To ensure survival of the tree and avoid trivialities, we assume that~$p_0 = 0$.
 To state our results later, we let~$\mu$ and~$\sigma^2$ be respectively the mean and the variance of the number of offspring per individual:
 $$ \mu = \sum_{k = 1}^{\infty} \,k p_k \quad \hbox{and} \quad \sigma^2 = \sum_{k = 1}^{\infty} \,(k - \mu)^2 p_k. $$
 To model the spread of information within the financial market, we use bidirectional bond percolation on the Galton-Watson tree.
 More precisely, we let~$p, q \in (0, 1)$ and assume that each of the edges of the tree is independently open
 $$ \begin{array}{rcl}
    \hbox{with probability~$p$} & \hbox{in the direction parent $\to$ offspring} \vspace*{4pt} \\
    \hbox{with probability~$q$} & \hbox{in the direction offspring $\to$ parent} \end{array} $$
 In other words, each edge is identified to two arrows.
 The arrow going away from the root is open with probability~$p$ whereas the arrow going towards the root is open with probability~$q$.
 Injecting the information at vertex~$x$, which represents the trader who first learns about the information, we are interested in the size~$S$ of the cluster of open edges~$\C$ starting at~$x$, i.e.,
 $$ S = \card (\C) \quad \hbox{where} \quad \C = \{y \in V : \hbox{there is a directed open path~$x \to y$} \}. $$
 In the terminology of percolation theory, we assume that a fluid injected at vertex~$x$ can only spread through the open edges so the vertices in cluster~$\C$ are called wet vertices.
 In the context of financial markets, vertex~$x$ represents the source of the information while the wet vertices are the traders who are aware of the information. \vspace*{5pt} \\
\noindent {\bf First moment.}
 To begin with, we look at the first moment of the cluster size~$S$, which represents the mean number of traders who are aware of the information.
 Although we model the financial market using an infinite Galton-Watson tree, we study the first moment on the truncated Galton-Watson tree with radius~$R$, defined as the subgraph of~$\T$ induced by the vertices at distance at most~$R$ from the root of the tree, in order to also extend a result in~\cite{jevtic_lanchier_2020} concerned with the symmetric case~$p = q$.
 By monotone convergence, the first moment on the infinite tree can be deduced by simply taking the limit as~$R \to\infty$.
 Partitioning the set of wet vertices into subtrees and using independence, we obtain the following theorem.
\begin{theorem}-- \
\label{th:first}
 The conditional first moment on the tree with radius~$R$ given that the information starts at distance~$r$ from the root is equal to
 $$ E_r (S) = \frac{1}{1 - \mu p} \bigg(1 + q \bigg(\frac{1 - q^r}{1 - q} \bigg) (1 - p) - (\mu p)^{R - r + 1} \bigg(\frac{1 - pq (1 + (\mu - 1)(\mu pq)^r)}{1 - \mu pq} \bigg) \bigg). $$
\end{theorem}
 It directly follows from the theorem that, in the supercritical phase~$\mu p > 1$, the first moment of the cluster size on the infinite Galton-Watson tree is infinite.
 In contrast, in the subcritical phase~$\mu p < 1$, the first moment on the infinite tree reduces to
 $$ E_r (S) = \frac{1}{1 - \mu p} \bigg(1 + q \bigg(\frac{1 - q^r}{1 - q} \bigg) (1 - p) \bigg). $$
 In addition, setting~$p = q$ in the theorem gives
 $$ E_r (S) = \frac{1}{1 - \mu p} \bigg(1 + p (1 - p^r) - (\mu p)^{R - r + 1} \bigg(\frac{1 - p^2 (1 + (\mu - 1)(\mu p^2)^r)}{1 - \mu p^2} \bigg) \bigg), $$
 which is exactly the expression found in~\cite[Theorem~4]{jevtic_lanchier_2020}, but we point out that, even though our result extends the result in~\cite{jevtic_lanchier_2020} to the asymmetric case, our approach leads to a much shorter and more elegant proof.
 More precisely, the proof in~\cite{jevtic_lanchier_2020} relies on a tedious combinatorial argument that consists in counting the number of open paths of a given length starting from the source of the information whereas our proof consists in finding the wet vertices along the path going from the source of the information to the root of the tree and then partitioning the cluster of wet vertices into (disjoint) subtrees starting from each of these vertices. \vspace*{5pt} \\
\noindent {\bf Second moment.}
 We now study the second moment of the cluster size on the infinite tree.
 It follows from Theorem~\ref{th:first} that, on the infinite Galton-Watson tree, all the higher moments are infinite in the supercritical phase so we assume that~$\mu p < 1$.
 The key to computing the second moment is independence.
 More precisely, writing again the cluster of wet vertices as a disjoint union of subtrees, the second moment can be computed explicitly using that the sizes of these subtrees are independent and identically distributed random variables, which gives the following result.
\begin{theorem}-- \
\label{th:second}
 Let~$\mu p < 1$.
 Then, the conditional second moment on the infinite tree given that the information starts at distance~$r$ from the root is equal to
 $$ \begin{array}{rcl}
    \displaystyle E_r (S^2) & \n = \n &
    \displaystyle \frac{1}{(1 - \mu p)^2} \ \bigg(1 + \frac{p (1 - p) \mu + p^2 \sigma^2}{1 - \mu p} \bigg) \vspace*{8pt} \\ && \hspace*{5pt} + \
    \displaystyle \bigg(1 + \frac{2 (1 + (\mu - 1) p)}{1 - \mu p} + \frac{2 (\mu - 1) p + p (1 - p)(\mu - 1) + p^2 \sigma^2 + (\mu - 1)^2 p^2}{(1 - \mu p)^2} \vspace*{8pt} \\ && \hspace*{100pt} + \
    \displaystyle \frac{(p (1 - p) \mu + p^2 \sigma^2)(\mu - 1) p}{(1 - \mu p)^3} \bigg) \ q \bigg(\frac{1 - q^r}{1 - q} \bigg) \vspace*{8pt} \\ && \hspace*{5pt} + \
    \displaystyle \bigg(1 + \frac{(\mu - 1) p}{1 - \mu p} \bigg)^2 \ \frac{2q^2 (1 - rq^{r - 1} + (r - 1) q^r)}{(1 - q)^2}. \end{array} $$
\end{theorem}
 Combining Theorems~\ref{th:first} and~\ref{th:second} implies that, in the subcritical phase, the tail distribution of the size of the cluster of wet vertices has a quadratic decay:
 there exists~$C = C (\mu p) < \infty$ such that
 $$ P (S > n) \leq C / n^2 \quad \hbox{for all} \quad n > 0. $$
 Indeed, according to the theorems, both the first and second moments are finite when~$\mu p < 1$, therefore it follows from Chebyshev's inequality that
 $$ P (S > n) \leq P (|S - E (S)| > |n - E (S)|) \leq \frac{\var (S)}{(n - E (S))^2} \leq \frac{E (S^2)}{(n - E (S))^2} $$
 for all~$n > E (S)$, which shows quadratic decay.
\newpage
\noindent {\bf Exponential decay of the diameter.}
 Another quantity of interest which also accounts for the geometry of the set of wet vertices is the diameter of the cluster~$\C$ defined as the maximum graph distance between any two wet vertices:
 $$ \diam (\C) = \max \,\{d (x, y) : x, y \in \C \}. $$
 In this case, studying the spread of information from a dynamical point of view starting from the highest wet vertex and moving one generation down the tree at each time step, we can prove an exponential decay.
 More precisely, we have the following theorem.
\begin{theorem}-- \
\label{th:decay}
 Let~$\mu p < 1$.
 Then, the conditional probability that the diameter is larger than~$2n$ given that the information starts at distance~$r$ from the root is
 $$ P_r (\diam (\C) \geq 2n) \leq \frac{1 - (q / \mu p)^{r + 1}}{1 - (q / \mu p)} \ (\mu p)^n \quad \hbox{for all} \quad n > r. $$
\end{theorem}
 The theorem indeed implies that, in the subcritical phase~$\mu p < 1$, the tail distribution of the diameter of the cluster of wet vertices, i.e., the maximum distance between any two traders who learn about the information, decays exponentially.


\section{Partition into disjoint subtrees}
 To get ready for the proofs of Theorems~\ref{th:first} and~\ref{th:second} in the next two sections, we first explain and study the partition of the set of wet vertices into disjoint subtrees.
 More precisely, we first study the distribution of the random number of subtrees and compute the first and second moment of the size of these subtrees.
 From now on, we assume that the information starts at a vertex~$x$ with~$d (0, x) = r$.
 By spherical symmetry, the specific choice of~$x$ is unimportant as long as the vertex is at distance~$r$ from the root.
 There is a unique directed path
 $$ x_0 = 0 \to x_1 \to x_2 \to \cdots \to x_{r - 1} \to x_r = x $$
 of length~$r$ going from the root to vertex~$x$ and we let
 $$ D = \max \,\{i = 0, 1, \ldots, r : x_{r - i} \ \hbox{is wet} \}. $$
 This is the distance between the source of the information and the highest wet vertex, and we refer the reader to Figure~\ref{fig:tree} for a picture.
 The following lemma gives preliminary results about the random variable~$D$ that will be useful later to prove the theorems.
\begin{figure}[t!]
\centering
\scalebox{0.65}{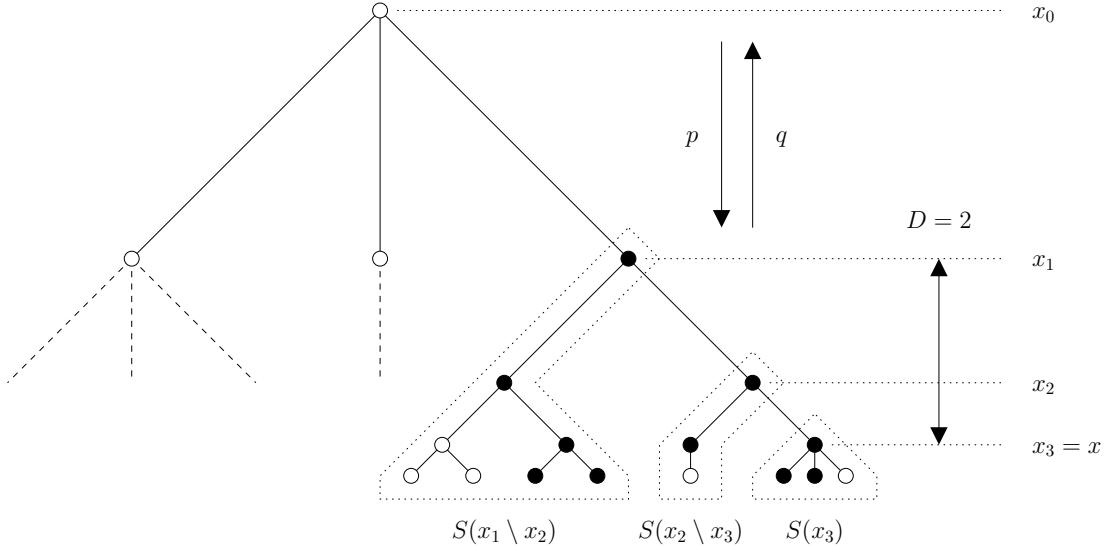}
\caption{\upshape{Picture of the partition into disjoint subtrees used to prove Theorems~\ref{th:first} and~\ref{th:second}.
 The black vertices represent the set of wet vertices while the white vertices are not wet.
 In our example, the information starts from~$x_3$ and spreads up to~$x_1$, which results in a partition of the cluster of wet vertices into three disjoint subtrees.
 Starting from the source of the information, the numbers of wet vertices in the subtrees are~3, 2 and 5, respectively.}}
\label{fig:tree}
\end{figure}
\begin{lemma}-- \
\label{lem:D}
 We have
 $$ E (D) = q \bigg(\frac{1 - q^r}{1 - q} \bigg) \quad \hbox{and} \quad E (D (D - 1)) = \frac{2q^2 (1 - rq^{r - 1} + (r - 1) q^r)}{(1 - q)^2}. $$
\end{lemma}
\begin{proof}
 Because the information spreads towards the root of the random tree with probability~$q$ independently through each of the edges, we have
 $$ P (D = k) = q^k (1 - q) \quad \hbox{for} \ k = 0, 1, \ldots, r - 1, \quad \hbox{and} \quad P (D = r) = q^r. $$
 In particular, using that
 $$ \sum_{k = 1}^{r - 1} \,k x^{k - 1} =
    \frac{\partial}{\partial x} \bigg(\sum_{k = 0}^{r - 1} x^k \bigg) =
    \frac{\partial}{\partial x} \bigg(\frac{1 - x^r}{1 - x} \bigg) = \frac{1 - rx^{r - 1} + (r - 1) x^r}{(1 - x)^2}, $$
 we deduce that the first moment is given by
 $$ \begin{array}{rcl}
      E (D) & \n = \n &
    \displaystyle \sum_{k = 0}^{r - 1} \,k q^k (1 - q) + r q^r = q (1 - q) \sum_{k = 1}^{r - 1} \,k q^{k - 1} + r q^r \vspace*{4pt} \\ & \n = \n &
    \displaystyle q (1 - q) \bigg(\frac{1 - rq^{r - 1} + (r - 1) q^r}{(1 - q)^2} \bigg) + r q^r = q \bigg(\frac{1 - q^r}{1 - q} \bigg). \end{array} $$
 Similarly, we have
 $$ \begin{array}{rcl}
    \displaystyle \sum_{k = 2}^{r - 1} \,k (k - 1) x^{k - 2} & \n = \n &
    \displaystyle \frac{\partial^2}{\partial x^2} \bigg(\sum_{k = 0}^{r - 1} x^k \bigg) =
    \displaystyle \frac{\partial}{\partial x} \bigg(\frac{1 - rx^{r - 1} + (r - 1) x^r}{(1 - x)^2} \bigg) \vspace*{8pt} \\ & \n = \n &
    \displaystyle \frac{2 - r (r - 1) x^{r - 2} + 2r (r - 2) x^{r - 1} - (r - 1)(r - 2) x^r}{(1 - x)^3} \end{array} $$
 from which it follows that
 $$ \begin{array}{rcl}
    \displaystyle E (D (D - 1)) & \n = \n &
    \displaystyle \sum_{k = 0}^{r - 1} \,k (k - 1) q^k (1 - q) + r (r - 1) q^r \vspace*{4pt} \\ & \n = \n &
    \displaystyle q^2 (1 - q) \ \frac{2 - r (r - 1) q^{r - 2} + 2r (r - 2) q^{r - 1} - (r - 1)(r - 2) q^r}{(1 - q)^3} \vspace*{4pt} \\ && \hspace*{50pt} + \
    \displaystyle q^2 \ \frac{r (r - 1) q^{r - 2} - 2r (r - 1) q^{r - 1} + r (r - 1) q^r}{(1 - q)^2} \vspace*{4pt} \\ & \n = \n &
    \displaystyle \frac{2q^2 (1 - rq^{r - 1} + (r - 1) q^r)}{(1 - q)^2}. \end{array} $$
 This completes the proof.
\end{proof} \\ \\
 Next, we define the random variables
\begin{equation}
\label{eq:subtrees}
  \begin{array}{rcl}
    S (x_j) & \n = \n & \hbox{number of wet vertices in the subtree starting at~$x_j$} \vspace*{4pt} \\
    S (x_j \setminus x_{j + 1}) & \n = \n & \hbox{number of wet vertices in the subtree starting at~$x_j$ but} \\ &&
                                             \hbox{excluding the subtree starting at~$x_{j + 1}$.} \end{array}
\end{equation}
 For instance, in the realization shown in Figure~\ref{fig:tree}, we have
 $$ S (x_3) = 3, \quad S (x_2 \setminus x_3) = 2, \quad S (x_1 \setminus x_2) = 5, \quad S (x_0 \setminus x_1) = 0. $$
 To estimate the first and second moments of~$S$ later, we now compute the first and second moments of the random variables in~\eqref{eq:subtrees}.
 To do this, let~$\xi_i = \bernoulli (p)$ be independent, and let~$Y$ be the offspring distribution, i.e., the random variable describing the random number of edges starting from a given vertex and going away from the root.
 In particular,
 $$ X_+ = \xi_1 + \xi_2 + \cdots + \xi_Y \quad \hbox{and} \quad X_- = \xi_1 + \xi_2 + \cdots + \xi_{Y - 1} $$
 are the random variables describing the number of wet offspring of a given vertex and the number of wet offspring of a given vertex excluding a given offspring, respectively.
 The next lemma gives the first and second moments of the first random variable in~\eqref{eq:subtrees}.
\begin{lemma}-- \
\label{lem:plus}
 Let~$\nu_+ = E (X_+)$ and~$\Sigma_+^2 = \var (X_+)$. Then,
 $$ \begin{array}{rcl}
     E (S (x_j)) & \n = \ & \displaystyle \frac{1 - \nu_+^{R - j + 1}}{1 - \nu_+} \vspace*{4pt} \\
     E (S (x_j)^2) & \n = \n & \displaystyle \frac{\Sigma_+^2}{(1 - \nu_+)^2} \ \bigg(\frac{1 - \nu_+^{2 (R - j) + 1}}{1 - \nu_+} - (2 (R - j) + 1) \nu_+^{R - j} \bigg) + \bigg(\frac{1 - \nu_+^{R - j + 1}}{1 - \nu_+} \bigg)^2. \end{array} $$
\end{lemma}
 For a detailed proof of this lemma, we refer to~\cite[Lemmas~8 and 9]{jevtic_lanchier_2020}.
 We now compute the first and second moments of the second set of random variables in~\eqref{eq:subtrees}.
\begin{lemma}-- \
\label{lem:minus}
 Let~$\nu_- = E (X_-)$ and~$\Sigma_-^2 = \var (X_-)$. Then,
 $$ \begin{array}{rcl}
      E (S (x_j \setminus x_{j + 1})) & \n = \n & 1 + \nu_- E (S (x_{j + 1})) \vspace*{4pt} \\
      E (S (x_j \setminus x_{j + 1})^2) & \n = \n & 1 + 2 \nu_- E (S (x_{j + 1})) + \nu_- E (S (x_{j + 1})^2) + (\Sigma_-^2 + \nu_-^2 - \nu_-) E (S (x_{j + 1}))^2. \end{array} $$
\end{lemma}
\begin{proof}
 Letting~$y_1, y_2, \ldots, y_X$ be the wet offspring of~$x_j$ other than~$x_{j + 1}$,
\begin{equation}
\label{eq:partition}
 S (x_j \setminus x_{j + 1}) = 1 + \sum_{n = 1}^X \,S (y_n) \quad \hbox{and} \quad X \overset{d}{=} X_-.
\end{equation}
 In particular, conditioning on~$X$, we get
 $$ E (S (x_j \setminus x_{j + 1})) = E \bigg(E \bigg(1 + \sum_{n = 1}^X \,S (y_n) \,\Big| \,X \bigg) \bigg) = 1 + E (X) E (S (y_n)) = 1 + \nu_- E (S (x_{j + 1})). $$
 Taking the square in~\eqref{eq:partition}
 $$ \begin{array}{rcl}
      S (x_j \setminus x_{j + 1})^2 & \n = \n &
    \displaystyle 1 + 2 \,\sum_{n = 1}^X S (y_n) + \bigg(\sum_{n = 1}^X S (y_n) \bigg)^2 \vspace*{4pt} \\ & \n = \n & 
    \displaystyle 1 + 2 \,\sum_{n = 1}^X S (y_n) + \sum_{n = 1}^X S (y_n)^2 + \sum_{n \neq m} S (y_n) S (y_m)  \end{array} $$
 then conditioning on~$X$ and using independence,
 $$ \begin{array}{rcl}
     E (S (x_j \setminus x_{j + 1})^2) & \n = \n &
     E (E (S (x_j \setminus x_{j + 1})^2 \,| \,X)) \vspace*{4pt} \\ & \n = \n &
     1 + 2 E (X) E (S (y_n)) + E (X) E (S (y_n)^2) + E (X (X - 1)) E (S (y_n))^2 \vspace*{4pt} \\ & \n = \n &
     1 + 2 \nu_- E (S (x_{j + 1})) + \nu_- E (S (x_{j + 1})^2) + (\Sigma_-^2 + \nu_-^2 - \nu_-) E (S (x_{j + 1}))^2. \end{array} $$
 This completes the proof.
\end{proof}


\section{Proof of Theorem \ref{th:first} (first moment)}
 Throughout this section, we assume that the network of interactions consists of the truncated Galton-Watson tree with radius~$R$ obtained by removing from the infinite tree all the vertices at distance more than~$R$ from the root.
 To prove Theorem~\ref{th:first}, the first step is to compute the mean of the random variables~$X_+$ and~$X_-$ defined above, which is done in the next lemma.
\begin{lemma}-- \
\label{lem:mean}
 For~$\nu_{\pm}$ defined as in Lemmas~\ref{lem:plus} and~\ref{lem:minus}, we have~$\nu_+ = \mu p$ and~$\nu_- = (\mu - 1) p$.
\end{lemma}
\begin{proof}
 Recalling the definition of~$X_{\pm}$ and~$Y$, and conditioning on~$Y$,
 $$ \begin{array}{rcl}
    \nu_+ & \n = \n & E (X_+) = E (E (X_+ \,| \,Y)) = E (E (\xi_1 + \cdots + \xi_Y \,| \,Y)) = E (Y) E (\xi_n) = \mu p \vspace*{4pt} \\
    \nu_- & \n = \n & E (X_-) = E (E (X_- \,| \,Y)) = E (E (\xi_1 + \cdots + \xi_{Y - 1} \,| \,Y)) = E (Y - 1) E (\xi_n) = (\mu - 1) p. \end{array} $$
 This completes the proof.
\end{proof} \\ \\
 Using the previous lemmas, we are now ready to prove Theorem~\ref{th:first}.
 To begin with, observe that, on the event~$D = k$, the total number of wet vertices can be written as
\begin{equation}
\label{eq:size}
  S = S (x_r) + \sum_{i = r - k}^{r - 1} S (x_i \setminus x_{i + 1}) = S (x_r) + \sum_{i = 1}^k \,S (x_{r - i} \setminus x_{r - i + 1}).
\end{equation}
 In particular, conditioning on~$D$ and using Lemma~\ref{lem:minus}, we get
\begin{equation}
\label{eq:thm1a}
  \begin{array}{rcl} E_r (S) & \n = \n &
  \displaystyle \sum_{k = 0}^r \,E (S \,| \,D = k) \,P (D = k) \vspace*{4pt} \\ & \n = \n &
  \displaystyle \sum_{k = 0}^r \,\bigg(E (S (x_r)) + \sum_{i = 1}^k \,E (S (x_{r - i} \setminus x_{r - i + 1})) \bigg) P (D = k) \vspace*{4pt} \\ & \n = \n &
  \displaystyle \sum_{k = 0}^r \,\bigg(E (S (x_r)) + \sum_{i = 0}^{k - 1} \ (1 + \nu_- E (S (x_{r - i})) \bigg) P (D = k) \vspace*{4pt} \\ & \n = \n &
  \displaystyle E (S (x_r)) + E (D) + \nu_- \ \sum_{k = 1}^r \ \sum_{i = 0}^{k - 1} \,E (S (x_{r - i})) \,P (D = k). \end{array}
\end{equation}
 Exchanging the two sums, and using Lemma~\ref{lem:plus}, we obtain
\begin{equation}
\label{eq:thm1b}
  \begin{array}{l}
  \displaystyle \sum_{k = 1}^r \ \sum_{i = 0}^{k - 1} \,E (S (x_{r - i})) \,P (D = k) =
  \displaystyle \sum_{i = 0}^{r - 1} \ \sum_{k = i + 1}^r \,E (S (x_{r - i})) \,P (D = k) \vspace*{4pt} \\ \hspace*{40pt} =
  \displaystyle \sum_{i = 0}^{r - 1} \,E (S (x_{r - i})) \,P (D > i) =
  \displaystyle \sum_{i = 0}^{r - 1} \,q^{i + 1} \,E (S (x_{r - i})) \vspace*{8pt} \\ \hspace*{40pt} =
  \displaystyle \sum_{i = 0}^{r - 1} \,q^{i + 1} \bigg(\frac{1 - \nu_+^{R - r + i + 1}}{1 - \nu_+} \bigg) =
  \displaystyle \frac{1}{1 - \nu_+} \bigg(q \ \sum_{i = 0}^{r - 1} \,q^i - q \,\nu_+^{R - r + 1} \ \sum_{i = 0}^{r - 1} \,(q \nu_+)^i \bigg) \vspace*{8pt} \\ \hspace*{40pt} =
  \displaystyle \frac{q}{1 - \nu_+} \bigg(\bigg(\frac{1 - q^r}{1 - q} \bigg) - \nu_+^{R - r + 1} \bigg(\frac{1 - (q \nu_+)^r}{1 - q \nu_+} \bigg) \bigg). \end{array}
\end{equation}
 Combining~\eqref{eq:thm1a} and~\eqref{eq:thm1b}, and using Lemmas~\ref{lem:D} and~\ref{lem:plus}, we deduce that
 $$ \begin{array}{rcl} E_r (S) & \n = \n &
    \displaystyle \frac{1 - \nu_+^{R - r + 1}}{1 - \nu_+} + q \bigg(\frac{1 - q^r}{1 - q} \bigg) + \frac{q \nu_-}{1 - \nu_+} \bigg(\bigg(\frac{1 - q^r}{1 - q} \bigg) - \nu_+^{R - r + 1} \bigg(\frac{1 - (q \nu_+)^r}{1 - q \nu_+} \bigg) \bigg) \vspace*{8pt} \\ & \n = \n &
    \displaystyle \frac{1 - \nu_+^{R - r + 1}}{1 - \nu_+} + q \bigg(\frac{1 - q^r}{1 - q} \bigg) \bigg(\frac{1 - \nu_+ + \nu_-}{1 - \nu_+} \bigg) + \frac{q \nu_-}{1 - \nu_+} \bigg(- \nu_+^{R - r + 1} \bigg(\frac{1 - (q \nu_+)^r}{1 - q \nu_+} \bigg) \bigg) \vspace*{8pt} \\ & \n = \n &
    \displaystyle \frac{1}{1 - \nu_+} \bigg(1 + q \bigg(\frac{1 - q^r}{1 - q} \bigg)(1 - \nu_+ + \nu_-) - \nu_+^{R - r + 1} \bigg(1 + q \nu_- \bigg(\frac{1 - (q \nu_+)^r}{1 - q \nu_+} \bigg) \bigg) \bigg). \end{array} $$
 Using Lemma~\ref{lem:mean} also gives~$1 - \nu_+ + \nu_- = 1 - p$ and
 $$ 1 + q \nu_- \bigg(\frac{1 - (q \nu_+)^r}{1 - q \nu_+} \bigg) = \frac{1 - pq (1 + (\mu - 1)(\mu pq)^r)}{1 - \mu pq}. $$
 In conclusion,
 $$ E_r (S) = \frac{1}{1 - \mu p} \bigg(1 + q \bigg(\frac{1 - q^r}{1 - q} \bigg) (1 - p) - (\mu p)^{R - r + 1} \bigg(\frac{1 - pq (1 + (\mu - 1)(\mu pq)^r)}{1 - \mu pq} \bigg) \bigg), $$
 which completes the proof of Theorem~\ref{th:first}.


\section{Proof of Theorem \ref{th:second} (second moment)}
 Throughout this section, we assume that the network of interactions consists of the infinite Galton-Watson tree~$\T$.
 To prove Theorem~\ref{th:second}, we follow the same strategy as for Theorem~\ref{th:first} and first compute the variance of the random variables~$X_+$ and~$X_-$, which is done in the next lemma.
\begin{lemma}-- \
\label{lem:variance}
 For~$\Sigma_{\pm}^2$ defined as in Lemmas~\ref{lem:plus} and~\ref{lem:minus}, we have
 $$ \Sigma_+^2 = p (1 - p) \mu + p^2 \sigma^2 \quad \hbox{and} \quad \Sigma_-^2 = p (1 - p)(\mu - 1) + p^2 \sigma^2. $$
\end{lemma}
\begin{proof}
 Conditioning on~$Y$ and using the law of total variance, we get
 $$ \begin{array}{rcl}
    \Sigma_+^2 = \var (X_+) & \n = \n & E (\var (X_+ \,| \,Y)) + \var (E (X_+ \,| \,Y)) \vspace*{4pt} \\
                            & \n = \n & E (p (1 - p) Y) + \var (pY) = p (1 - p) \mu + p^2 \sigma^2. \end{array} $$
 Similarly, for the variance of~$X_-$,
 $$ \begin{array}{rcl}
    \Sigma_-^2 = \var (X_-) & \n = \n & E (\var (X_- \,| \,Y)) + \var (E (X_- \,| \,Y)) \vspace*{4pt} \\
                            & \n = \n & E (p (1 - p)(Y - 1)) + \var (p (Y - 1)) = p (1 - p)(\mu - 1) + p^2 \sigma^2. \end{array} $$
 This completes the proof.
\end{proof} \\ \\
 We are now ready to prove Theorem~\ref{th:second}.
 Taking the square in~\eqref{eq:size}
 $$ \begin{array}{rcl} S^2 & \n = \n &
    \displaystyle S (x_r)^2 + 2 S (x_r) \ \sum_{i = 1}^k \,S (x_{r - i} \setminus x_{r - i + 1}) + \bigg(\sum_{i = 1}^k \,S (x_{r - i} \setminus x_{r - i + 1}) \bigg)^2 \vspace*{4pt} \\ & \n = \n &
    \displaystyle S (x_r)^2 + 2 S (x_r) \ \sum_{i = 1}^k \,S (x_{r - i} \setminus x_{r - i + 1}) \vspace*{0pt} \\ && \hspace*{25pt} + \
    \displaystyle \sum_{i = 1}^k \,S (x_{r - i} \setminus x_{r - i + 1})^2 + \sum_{i \neq j} \,S (x_{r - i} \setminus x_{r - i + 1}) S (x_{r - j} \setminus x_{r - j + 1}) \end{array} $$
 then conditioning on the event~$D = k$ and using independence of the random variables in~\eqref{eq:subtrees}~(because they represent the number of wet vertices in disjoint subtrees), we obtain
\begin{equation}
\label{eq:second}
  \begin{array}{rcl}
  \displaystyle E_r (S^2) = \sum_{k = 0}^r \bigg(E (S (x_r)^2) & \n + \n &
  \displaystyle 2 E (S (x_r)) \ \sum_{i = 1}^k \,E (S (x_{r - i} \setminus x_{r - i + 1})) \vspace*{-4pt} \\ & \n + \n &
  \displaystyle \sum_{i = 1}^k \,E (S (x_{r - i} \setminus x_{r - i + 1})^2) \vspace*{0pt} \\ & \n + \n &
  \displaystyle \sum_{i \neq j} \,E (S (x_{r - i} \setminus x_{r - i + 1})) \,E (S (x_{r - j} \setminus x_{r - j + 1})) \bigg) P (D = k). \end{array}
\end{equation}
 Using the monotone convergence theorem, taking the limit as~$R \to \infty$ in (and simplifying) the expressions in Lemma~\ref{lem:plus}, we deduce that, on the infinite Galton-Watson tree,
\begin{equation}
\label{eq:infinite-1}
  E (S (x_r)) = \frac{1}{1 - \nu_+} \quad \hbox{and} \quad E (S (x_r)^2) = \frac{1}{(1 - \nu_+)^2} \ \bigg(1 + \frac{\Sigma_+^2}{1 - \nu_+} \bigg).
\end{equation}
 This, together with Lemma~\ref{lem:minus}, implies that
\begin{equation}
\label{eq:infinite-2}
  \begin{array}{rcl}
    E (S (x_{r - i} \setminus x_{r - i + 1})) & \n = \n &
  \displaystyle 1 + \frac{\nu_-}{1 - \nu_+} \vspace*{4pt} \\
    E (S (x_{r - i} \setminus x_{r - i + 1})^2) & \n = \n &
  \displaystyle 1 + \frac{2 \nu_-}{1 - \nu_+} + \frac{\nu_-}{(1 - \nu_+)^2} \ \bigg(1 + \frac{\Sigma_+^2}{1 - \nu_+} \bigg) + \frac{\Sigma_-^2 + \nu_-^2 - \nu_-}{(1 - \nu_+)^2} \vspace*{8pt} \\ & \n = \n &
  \displaystyle 1 + \frac{2 \nu_-}{1 - \nu_+} + \frac{\Sigma_-^2 + \nu_-^2}{(1 - \nu_+)^2} + \frac{\Sigma_+^2 \nu_-}{(1 - \nu_+)^3}. \end{array}
\end{equation}
 Using that, in the limit as~$R \to \infty$, the terms in the two sums over~$i$ and the terms in the sum over~$i \neq j$ in equation~\eqref{eq:second} are constant, and using~\eqref{eq:infinite-1} and~\eqref{eq:infinite-2}, we obtain
 $$ \begin{array}{rcl}
    \displaystyle E_r (S^2) & \n = \n &
    \displaystyle \frac{1}{(1 - \nu_+)^2} \ \bigg(1 + \frac{\Sigma_+^2}{1 - \nu_+} \bigg) + \frac{2}{1 - \nu_+} \bigg(1 + \frac{\nu_-}{1 - \nu_+} \bigg) E (D) \vspace*{8pt} \\ && \hspace*{5pt} + \
    \displaystyle \bigg(1 + \frac{2 \nu_-}{1 - \nu_+} + \frac{\Sigma_-^2 + \nu_-^2}{(1 - \nu_+)^2} + \frac{\Sigma_+^2 \nu_-}{(1 - \nu_+)^3} \bigg) E (D) + \bigg(1 + \frac{\nu_-}{1 - \nu_+} \bigg)^2 E (D (D - 1)) \vspace*{8pt} \\ & \n = \n &
    \displaystyle \frac{1}{(1 - \nu_+)^2} \ \bigg(1 + \frac{\Sigma_+^2}{1 - \nu_+} \bigg) +
    \displaystyle \bigg(1 + \frac{2 (1 + \nu_-)}{1 - \nu_+} + \frac{2 \nu_- + \Sigma_-^2 + \nu_-^2}{(1 - \nu_+)^2} + \frac{\Sigma_+^2 \nu_-}{(1 - \nu_+)^3} \bigg) E (D) \vspace*{8pt} \\ && \hspace*{5pt} + \
    \displaystyle \bigg(1 + \frac{\nu_-}{1 - \nu_+} \bigg)^2 E (D (D - 1)). \end{array} $$
 Then, using Lemma~\ref{lem:D}, we get
 $$ \begin{array}{rcl}
    \displaystyle E_r (S^2) & \n = \n &
    \displaystyle \frac{1}{(1 - \nu_+)^2} \ \bigg(1 + \frac{\Sigma_+^2}{1 - \nu_+} \bigg) \vspace*{8pt} \\ && \hspace*{5pt} + \
    \displaystyle \bigg(1 + \frac{2 (1 + \nu_-)}{1 - \nu_+} + \frac{2 \nu_- + \Sigma_-^2 + \nu_-^2}{(1 - \nu_+)^2} + \frac{\Sigma_+^2 \nu_-}{(1 - \nu_+)^3} \bigg) \ q \bigg(\frac{1 - q^r}{1 - q} \bigg) \vspace*{8pt} \\ && \hspace*{5pt} + \
    \displaystyle \bigg(1 + \frac{\nu_-}{1 - \nu_+} \bigg)^2 \ \frac{2q^2 (1 - rq^{r - 1} + (r - 1) q^r)}{(1 - q)^2}. \end{array} $$
 and finally Lemmas~\ref{lem:mean} and~\ref{lem:variance},
 $$ \begin{array}{rcl}
    \displaystyle E_r (S^2) & \n = \n &
    \displaystyle \frac{1}{(1 - \mu p)^2} \ \bigg(1 + \frac{p (1 - p) \mu + p^2 \sigma^2}{1 - \mu p} \bigg) \vspace*{8pt} \\ && \hspace*{5pt} + \
    \displaystyle \bigg(1 + \frac{2 (1 + (\mu - 1) p)}{1 - \mu p} + \frac{2 (\mu - 1) p + p (1 - p)(\mu - 1) + p^2 \sigma^2 + (\mu - 1)^2 p^2}{(1 - \mu p)^2} \vspace*{8pt} \\ && \hspace*{100pt} + \
    \displaystyle \frac{(p (1 - p) \mu + p^2 \sigma^2)(\mu - 1) p}{(1 - \mu p)^3} \bigg) \ q \bigg(\frac{1 - q^r}{1 - q} \bigg) \vspace*{8pt} \\ && \hspace*{5pt} + \
    \displaystyle \bigg(1 + \frac{(\mu - 1) p}{1 - \mu p} \bigg)^2 \ \frac{2q^2 (1 - rq^{r - 1} + (r - 1) q^r)}{(1 - q)^2}, \end{array} $$
 which completes the proof of Theorem~\ref{th:second}.


\section{Proof of Theorem \ref{th:decay} (exponential decay)}
 To prove exponential decay of the diameter of the cluster of wet vertices, the idea is to study the process that keeps track of the random number of wet vertices at distance~$n$ from the highest wet vertex in the tree.
 More precisely, on the event~$D = k$, we let
 $$ X_n = \card (\C_n) \quad \hbox{where} \quad \C_n = \{x \in \C : d (x, x_{r - k}) = n \}. $$
 Recall that~$x_{r - k}$ is the unique vertex along the path connecting the source of the information and the root of the tree that is at distance~$k$ of the source of the information.
 The next lemma shows that, in the subcritical phase~$\mu p < 1$, the expected value of the process decays exponentially.
\begin{lemma}-- \
\label{lem:decay}
 Given that the information starts at distance~$r$ from the root,
 $$ E_r (X_{n + 1}) = \left\{\begin{array}{ccl}
                      \mu p \,E (X_n) + (1 - p) \,q^{n + 1} & \hbox{for} & n < r \vspace*{2pt} \\
                      \mu p \,E (X_n)                       & \hbox{for} & n \geq r. \end{array} \right. $$
\end{lemma}
\begin{proof}
 Recall that each vertex produces~$\mu$ offspring in average and that each of the offspring of a wet vertex is wet independently with probability~$p$, from which it follows that each wet vertex has~$\mu p$ wet offspring in average.
 Now, given~$D = k$, the process is conditioned so that
 $$ x_{r - k} \in \C_0, \quad x_{r - k + 1} \in \C_1, \quad \ldots \quad x_{r - 1} \in \C_{k - 1} \quad \hbox{and} \quad  x_r \in \C_k $$
 are wet therefore, until generation~$k - 1$, the vertices in~$\C_n$ all have~$\mu p$ wet offspring in average except for vertex~$x_{r - k + n}$ that has~$(\mu - 1) p + 1$ wet offspring in average.
 This implies that
 $$ E_r (X_{n + 1} \,| \,X_n, D = k) = \left\{\begin{array}{ccl}
                                       \mu p (X_n - 1) + (\mu - 1) p + 1 & \hbox{for} & n < k \vspace*{2pt} \\
                                       \mu p X_n     & \hbox{for} & n \geq k. \end{array} \right. $$
 Conditioning on the random variable~$D$, we deduce that
 $$ \begin{array}{rcl}
     E_r (X_{n + 1} \,| \,X_n) & \n = \n &
    \displaystyle \sum_{k = 0}^{\infty} \,E_r (X_{n + 1} \,| \,X_n, D = k) P_r (D = k) \vspace*{4pt} \\ & \n = \n &
    \displaystyle \sum_{k = 0}^n \ (\mu p X_n) \,P_r (D = k) + \sum_{k = n + 1}^{\infty} (\mu p (X_n - 1) + (\mu - 1) p + 1) \,P_r (D = k) \vspace*{4pt} \\ & \n = \n &
    \displaystyle \sum_{k = 0}^n \ (\mu p X_n) \,P_r (D = k) + \sum_{k = n + 1}^{\infty} (\mu p X_n + 1 - p) \,P_r (D = k) \vspace*{4pt} \\ & \n = \n &
    \displaystyle \sum_{k = 0}^{\infty} \ (\mu p X_n) \,P_r (D = k) + (1 - p) \sum_{k = n + 1}^{\infty} P_r (D = k) =
    \mu p X_n + (1 - p) P_r (D > n). \end{array} $$
 Recalling the probability mass function of~$D$, we deduce that
 $$ E_r (X_{n + 1}) = E (E (X_{n + 1} \,| \,X_n)) = \left\{\begin{array}{ccl}
                                                    \mu p \,E (X_n) + (1 - p) \,q^{n + 1} & \hbox{for} & n < r \vspace*{2pt} \\
                                                    \mu p \,E (X_n)                       & \hbox{for} & n \geq r. \end{array} \right. $$
 This completes the proof.
\end{proof} \\ \\
 It follows from the lemma that, for all~$n \leq r$,
 $$ \begin{array}{rcl}
     E_r (X_n) & \n = \n & (\mu p) \,E_r (X_{n - 1}) + (1 - p) \,q^n \vspace*{4pt} \\
               & \n = \n & (\mu p)^2 E_r (X_{n - 2}) + (\mu p) (1 - p) \,q^{n - 1} + (1 - p) \,q^n \vspace*{4pt} \\
               & \n = \n & (\mu p)^3 E_r (X_{n - 3}) + (\mu p)^2 (1 - p) \,q^{n - 2} + (\mu p)(1 - p) \,q^{n - 1} + (1 - p) \,q^n \vspace*{4pt} \\
               & \n = \n & (\mu p)^n E_r (X_0) + (1 - p)((\mu p)^{n - 1} \,q + (\mu p)^{n - 2} \,q^2 + \cdots + (\mu p) \,q^{n - 1} + q^n) \vspace*{4pt} \\
               & \n \leq \n & (\mu p)^n E_r (X_0) + (\mu p)^{n - 1} \,q + (\mu p)^{n - 2} \,q^2 + \cdots + (\mu p) \,q^{n - 1} + q^n. \end{array} $$
 Then, using that~$E_r (X_0) = 1$, we get
 $$ E_r (X_n) \leq \sum_{k = 0}^n \,(\mu p)^{n - k} q^k = (\mu p)^n \,\sum_{k = 0}^n \bigg(\frac{q}{\mu p} \bigg)^k = \frac{1 - (q / \mu p)^{n + 1}}{1 - (q / \mu p)} \ (\mu p)^n \quad \hbox{for all} \quad n \leq r. $$
 Observing also that, for all~$n > r$,
 $$ E_r (X_n) \leq (\mu p) \,E_r (X_{n - 1}) \leq (\mu p)^2 E_r (X_{n - 2}) \leq \cdots \leq (\mu p)^{n - r} E_r (X_r) $$
 we conclude that, for all~$n > r$, the diameter exceeds~$2n$ with probability
 $$ \begin{array}{rcl} P_r (\diam (\C) \geq 2n) & \n \leq \n &
    \displaystyle P_r (X_n > 0) = \sum_{k = 1}^{\infty} \,P_r (X_n = k) \leq \sum_{k = 1}^{\infty} \,k \,P_r (X_n = k) = E_r (X_n) \vspace*{8pt} \\ & \n \leq \n &
    \displaystyle (\mu p)^{n - r} E_r (X_r) \leq (\mu p)^{n - r} \ \frac{1 - (q / \mu p)^{r + 1}}{1 - (q / \mu p)} \ (\mu p)^r =
    \displaystyle \frac{1 - (q / \mu p)^{r + 1}}{1 - (q / \mu p)} \ (\mu p)^n. \end{array} $$
 This completes the proof of Theorem~\ref{th:decay}.



 

\bibliographystyle{plain}
\bibliography{biblio.bib}

\end{document}